\documentclass[12pt]{article}
\usepackage{fullpage,graphicx,psfrag,amsmath,amsfonts, verbatim}
\usepackage{mystyle}
\usepackage[small,bf]{caption}
\usepackage{xcolor}
\usepackage{amsthm}
\usepackage{hyperref}
\usepackage{upgreek}
\hypersetup{
    colorlinks=true,
    linkcolor=blue,
    filecolor=magenta,      
    urlcolor=cyan,
    pdftitle={Overleaf Example},
    pdfpagemode=FullScreen,
    }
\newtheorem{lemma}{Lemma}
\usepackage{amssymb}
\usepackage[]{mdframed}
\newtheorem*{thm*}{Theorem}
\newtheorem*{rem*}{Remark}
\newtheorem*{prop*}{Proposition}
\newtheorem*{lemma*}{Lemma}
\usepackage[left=2.4cm, right=2.4cm]{geometry}

\usepackage{mathtools} 

\renewcommand{\argmin}{\mathop{\mathrm{argmin}}}

\renewcommand{\argmax}{\mathop{\mathrm{argmax}}}

\newcommand{\diam}{\mathrm{diam}}

\newcommand{\as}{\alpha}
\newcommand{\bs}{\beta}
\newcommand{\pis}{\pi}
\newcommand{\cs}{c}
\newcommand{\KL}{\mathrm{KL}}

\newcommand{\gabriel}[1]{{\color{black}#1}}

\newtheorem{innerassumption}{Assumption}
\newenvironment{assumption}
  {\innerassumption}
  {\endinnerassumption}

\author{
\gabriel{Francisco Andrade}
\thanks{\gabriel{INRIA (PreMeDICaL \& HeKA), 
\texttt{francisco.de-lima-andrade@inria.fr}}}
\and
\gabriel{Gabriel Peyr\'e}
\thanks{\gabriel{CNRS and ENS, PSL Universit\'e,
\texttt{gabriel.peyre@ens.fr}}}
\and
\gabriel{Clarice Poon}
\thanks{\gabriel{Mathematics Institute, University of Warwick,
\texttt{clarice.poon@warwick.ac.uk}}}
}

\title{Sample complexity of unbalanced entropic OT}

\date{}
\begin{document}
\maketitle

\begin{abstract}
\gabriel{Optimal transport (OT) has become a central language for comparing
probability measures, but exact balanced OT is often both too rigid for data
with missing, created, or destroyed mass and subject to unfavorable high-dimensional sample complexity.  Entropic regularization and unbalanced relaxations address these
limitations in complementary ways. Entropy smooths the geometry, improves
statistical behavior, and enables fast Sinkhorn-type algorithms, while
unbalanced marginal penalties replace hard conservation constraints by
divergence terms adapted to noisy empirical data. This paper studies the
sample complexity of entropic unbalanced OT at the level of the optimal
coupling, rather than only the scalar transport value. We develop a
translation-invariant dual formulation, prove compactness and strong convexity
properties for the intrinsic dual variables, and convert these geometric
estimates into high-probability finite-sample bounds for empirical couplings.
The results clarify why regularization is a practical necessity in machine
learning applications: it softens the curse of dimensionality, reduces the
number of samples needed for stable transport estimation, and keeps the
resulting estimators compatible with scalable Sinkhorn-type solvers.}
\end{abstract}


\section{Introduction}
%

Among the various formulations of unbalanced optimal transport, entropic regularized models play a prominent role due to their computational tractability and smooth dual structure. In this setting, the primal problem admits a dual formulation in terms of generalized Kantorovich potentials, and the optimal coupling is expressed through an exponential scaling of the cost. While this structure closely parallels the balanced case, a crucial geometric difference emerges: the dual objective generally loses the translation (gauge) invariance that characterizes classical entropic OT.
In the balanced case, gauge freedom allows one to quotient out a flat direction in the dual, yielding a strictly convex objective on a suitable subspace. In contrast, for unbalanced OT,  the dual functional may fail to be strongly convex, and its geometry depends sensitively on the chosen divergences. As a result, controlling the stability of dual potentials and deriving statistical guarantees for empirical estimators becomes significantly more delicate.
The goal of this paper is to show that, under very mild and largely divergence-agnostic assumptions, one can recover a favorable geometric structure for the dual problem and leverage it to obtain finite-sample guarantees for unbalanced entropic OT for the transport plan. To the best of our knowledge, existing statistical guarantees for unbalanced entropic OT mostly concern the objective value (or related stability statements), rather than high-probability finite-sample bounds directly on the optimal coupling/plan. 

\subsection{Previous works}

\paragraph{\gabriel{Entropic optimal transport.}}
\gabriel{Optimal transport has become a central tool in statistics and machine learning because it compares probability measures while respecting the geometry of the data space; a broad computational overview is given in \cite{peyre2019computational}. It has been used as a Wasserstein loss for structured prediction \cite{frogner2015learning}, as a Sinkhorn-divergence loss for generative modeling \cite{genevay2018learning}, and as a debiased interpolation between entropic OT and maximum mean discrepancies for distributional comparison \cite{feydy2019interpolating}. The classical, unregularized Wasserstein distance has however a well-known statistical drawback: empirical convergence rates deteriorate with the ambient dimension, hence a curse of dimensionality \cite{fournier2015rate}. Entropic regularization mitigates this difficulty while making OT computationally usable at scale. On the algorithmic side, matrix-scaling iterations provide the basic normalization principle \cite{Sinkhorn64}, entropic OT turns this principle into a large-scale learning tool \cite{CuturiSinkhorn}, and regularized transport solvers can be interpreted as iterative Bregman projections \cite{Benamou2015}. These iterations involve kernel evaluations and reductions, which are highly parallelizable and therefore well suited to GPU implementations. On the statistical side, sample-complexity bounds are available for Sinkhorn divergences \cite{genevay2019sample}, central limit theorems and statistical bounds have been established for entropic OT costs \cite{mena2019statistical}, and sharper plan-level stability estimates follow from the strong convexity of the entropic dual modulo translations \cite{rigollet2022sample}.}

\paragraph{\gabriel{Unbalanced optimal transport.}}
\gabriel{Unbalanced optimal transport (UOT) extends classical optimal transport by relaxing marginal constraints through divergence penalties, thereby allowing for mass creation and destruction. Early examples include free-boundary formulations of partial transport \cite{caffarelli2010free} and the analysis of the optimal partial transport problem \cite{figalli2010optimal}. Distances between finite Radon measures remove the equal-mass requirement directly \cite{kondratyev2016new}. Dynamic and geometric formulations later emerged through the Hellinger--Kantorovich and entropy-transport framework \cite{liero2018optimal}, as well as through dynamic and Kantorovich formulations for unbalanced OT \cite{chizat2018unbalanced}. Entropic regularization also leads to scalable UOT solvers: scaling algorithms were developed for unbalanced problems \cite{chizat2018scaling}, and Sinkhorn divergences were adapted to the unbalanced setting \cite{sejourne2019sinkhorn}. A broader account of the theory, numerics, and applications of UOT can be found in \cite{sejourne2023unbalanced}.}
\gabriel{These developments are driven by applications in which the coupling itself is often the relevant object. Wasserstein-type losses appear in learning problems \cite{frogner2015learning}, while Sinkhorn divergences provide losses for generative modeling \cite{genevay2018learning}. Regularized OT has been used to infer developmental trajectories from single-cell data \cite{schiebinger2019optimal}, proximal OT models have been proposed for population dynamics \cite{bunne2022proximal}, and regularized unbalanced OT has been used to learn stochastic dynamics from snapshots \cite{zhang2024learning}. In a related direction, unbalanced diffusion Schr\"odinger bridges combine mass variation with stochastic interpolation \cite{pariset2023unbalanced}. In these settings, finite-sample control of the plan is more informative than value convergence alone, since the coupling is used for prediction, interpolation, or downstream learning.}

\paragraph{\gabriel{Dual formulations and translation invariance.}}

\gabriel{In balanced entropic optimal transport, the dual objective is invariant under additive translations of the Kantorovich potentials, reflecting the fact that the dual variables are identifiable only up to a one-dimensional affine subspace. This non-uniqueness is routinely handled in Sinkhorn scaling by explicit normalization of the dual iterates \cite{CuturiSinkhorn}. At the level of analysis, the balanced entropic dual becomes strongly convex only after quotienting out this translation direction \cite{rigollet2022sample}.}
\gabriel{In contrast, this symmetry is generally broken in unbalanced optimal transport because the marginal divergence penalties depend on the absolute levels of the potentials. This loss of translation invariance has important algorithmic consequences and has motivated explicit normalization strategies. Translation-invariant Sinkhorn iterations for UOT improve numerical stability and convergence speed \cite{sejourne2022translation}, and related normalization ideas already appear in unbalanced scaling algorithms \cite{chizat2018scaling}.}

\paragraph{\gabriel{Statistical guarantees and sample complexity.}}
\gabriel{The statistical properties of entropic optimal transport have been extensively studied in the balanced setting. Finite-sample bounds for Sinkhorn divergences can be obtained through empirical-process arguments \cite{genevay2019sample}. Statistical bounds and central limit theorems are available for entropic OT costs \cite{mena2019statistical}. These results should be contrasted with the unregularized Wasserstein rates, which display the high-dimensional curse of dimensionality \cite{fournier2015rate}. Beyond scalar cost functionals, stability properties of entropic dual potentials and the associated couplings can also be controlled in the balanced setting \cite{rigollet2022sample}.}
\gabriel{For unbalanced optimal transport, statistical analyses are comparatively less developed. Unbalanced Sinkhorn divergences come with stability and computational guarantees \cite{sejourne2019sinkhorn}, and the dynamic and Kantorovich formulations provide the basic variational framework \cite{chizat2018unbalanced}. General sample-complexity results for the transport plan itself, especially for broad $\varphi$-divergences, remain largely absent. The present work addresses this gap by combining a translation-invariant dual formulation with uniform compactness and strong convexity to derive high-probability sample-complexity bounds for unbalanced entropic optimal transport. In contrast to prior work, our results apply to a broad class of unbalanced models and explicitly characterize the role of the translation direction in controlling statistical fluctuations.}

\subsection{Contribution}

\gabriel{
Section~\ref{sec:translation_invariant_formulation} develops the main geometric contribution. It introduces the envelope $\Tt^{\alpha,\beta}$, which removes the scalar translation direction from the intrinsic dual variables. Theorem~\ref{thm:compact_anchor} shows that this envelope has a compact anchored minimization domain $\Pp_0^{\alpha,\beta}$ and is strongly convex there in the weighted $L^2(\alpha,\beta)$ norm. This provides the dual stability mechanism used in Proposition~\ref{prop:convergence_primal_weak} to obtain convergence of anchored potentials and weak convergence of the associated transport plans.

Subsection~\ref{subsec:uniform_translation_compactness} controls the remaining scalar translation. Under the domain condition in Assumption~\ref{ass:bounded}, Theorem~\ref{thm:restrictCompactSets} proves that the envelope minimization can be restricted to a fixed interval $[-R,R]$, with the explicit radius \eqref{eq:R_explicit}. This yields uniform $L^\infty$ bounds on translated dual potentials and applies to the usual unbalanced divergences, including KL, $\chi^2$, Hellinger, Jensen--Shannon, and $\alpha$-divergences.

Section~\ref{sec:sample_complexity} turns these compactness and convexity estimates into a plan-level statistical guarantee. The curvature and subgradient assumptions in Assumption~\ref{ass:strong_conv}, together with the translation stability estimate of Proposition~\ref{prop:lambda_stability}, allow us to compare empirical and population dual solutions. Theorem~\ref{SampleComplextiyuOTTHeorem} then proves the main high-probability bound \eqref{eq:samp-comp-forward} for the deviation between the empirical optimal plan $\pi^n$ and the population plan $\pi^\star$, using the concentration estimates collected in Section~\ref{sec:concentration_bounds}.}

\section{Unbalanced entropic OT}
In this section, we recall the formulation of entropically regularized unbalanced optimal transport and its dual representation. We work in a general setting where deviations from prescribed marginals are penalized via $\varphi$-divergences, encompassing most models used in practice.


\begin{assumption}\label{ass:compact}
The spaces $\Xx$ and $\Yy$ are compact. Let $c:\mathcal{X}\times\Yy\to \RR$ be a $L$-Lipschitz function.
\end{assumption}

\paragraph{Divergences}
The $\varphi$-divergence associated to a convex, positive, lower-semi-continuous function $\varphi:(0,\infty)\to [0,\infty)$ that satisfies $\varphi(1)=0$ is defined as 
\begin{align*}
D_\varphi(\alpha|\beta):=\int_\mathcal{X} \varphi\Big(\frac{d\alpha}{d\beta}(x)\Big)d\beta(x)+\varphi^\prime_\infty\int_\mathcal{X}d\alpha^\perp(x),
\end{align*} 	
where $\varphi^\prime_\infty=\lim_{x\to \infty}\varphi(x)/x$, and where, for any pair $(\alpha,\beta)\in \mathcal{M}^+(\mathcal{X})$, $\alpha^\perp$ is defined via the Radon-Nikodym-Lebesgue decomposition $\alpha=(d\alpha/d\beta)\beta+\alpha^\perp$ (see \cite{sejourne2019sinkhorn}). The Kullback-Leibler divergence is abbreviated by KL and corresponds to $D_\varphi$ with $\varphi(p)=p\log(p)-p+1$.  
The indication function of a set $S$ is denoted by $\iota_S$ and the notation $\iota_{=\beta}$ abbreviates the $D_{\lbrace 1\rbrace}$-divergence,  \emph{i.e.,} it is defined by $\iota_{=\beta}(\alpha)=0$ if $\alpha=\beta$ and $+\infty$ otherwise.

\paragraph{Primal problem}
Given $\alpha\in\Mm_+(\Xx)$, $\beta\in\Mm_+(\Yy)$ and a cost $c\in\Cc(\Xx\times\Yy)$,
the entropically regularized unbalanced OT coupling is defined by
\begin{equation}\label{eq:UOT2}
\pis
=
\argmin_{\pi \in \Mm_+(\Xx \times \Yy)}
\langle c, \pi \rangle
+\eta \; \KL(\pi \mid \alpha \otimes \beta)
+ D_{\varphi_1}(\pi_1 \mid \alpha)
+ D_{\varphi_2}(\pi_2 \mid \beta) .
\tag{p-UOT}
\end{equation}
Here $\pi_1,\pi_2$ denote the marginals of $\pi$, and $\varphi_1,\varphi_2$ are $\varphi$-divergences. We denote the masses by $m_\alpha \eqdef \alpha(\Xx)$ and $m_\beta \eqdef \beta(\Yy)$.

Balanced OT is recovered by taking hard constraints on marginals, that is, $\varphi_i = \iota_{\{1\}}$ and $m_\alpha=m_\beta = 1$, whereas other divergences
allow mass variation.

\paragraph{Dual problem}
Under Assumption~\ref{ass:compact} and the standing assumptions that $\varphi_i$ are
proper convex, positive, lower-semicontinuous entropy functions with $\varphi_i(1)=0$,
the primal problem admits a minimizer by the Fenchel--Rockafellar duality argument
of \cite[Theorem~3.2]{chizat2018scaling}; this minimizer is unique because of the
strict convexity of the entropic KL term. Moreover, the same duality argument gives
the standard dual characterization recalled next; see also
\cite[Sec.~3]{sejourne2019sinkhorn} for the entropic unbalanced formulation under
these entropy-function assumptions. The Kantorovich potentials themselves need not
be unique before fixing the translation/gauge freedom; uniqueness of an anchored
representative requires additional compactness/coercivity or quotient arguments
\cite[Sec.~3.3]{sejourne2019sinkhorn}, as addressed below through the anchored
envelope formulation.

There exist Kantorovich potentials $(f^\star,g^\star)$ such that
\begin{equation}\label{RadonDensity}
\frac{d\pis}{d\as\otimes\bs}
=
\exp\!\Big(\frac{f^\star \oplus g^\star - \cs}{\eta}\Big)
\quad
(\as\otimes\bs)\text{-a.e.},
\end{equation}
where $(f^\star,g^\star)$ solve the dual problem
\begin{equation}\label{eq:UOT2_dual}
(f^\star,g^\star)\in
\argmin_{(f,g)\in L^2(\as)\times L^2(\bs)}
\Kk^{\as,\bs}(f,g),
\tag{d-UOT}
\end{equation}
with
\begin{equation}\label{eq:kantorovich_fix_cost}
\Kk^{\as,\bs}(f,g)
:=
\langle \varphi_1^\ast(-f),\as\rangle
+
\langle \varphi_2^\ast(-g),\bs\rangle
+
\eta \Big\langle \exp\!\Big(\frac{f\oplus g-\cs}{\eta} \Big),\as\otimes\bs\Big\rangle.
\end{equation}

In the classical OT setting where $m_\al = m_\beta$, and $\varphi_i^*(x) = x$ for $i=1,2$, and hence, 
$\Kk^{\alpha,\beta}$ is invariant under the translation
\[
(f,g)\mapsto(f+\lambda,g-\lambda),\qquad \lambda\in\mathbb R.
\]
Here, translation invariance of the dual objective allows one to quotient out a one-dimensional flat direction, yielding a strictly convex functional on a suitable subspace. 
This geometric property plays a central role in stability analyses and convergence guarantees of the transport plan.
In the unbalanced case, although the plan is invariant to scalar translation of $f^\star$ and $g^\star$, the translation invariant property is lost in $\Kk^{\alpha,\beta}$.

\section{Translational invariation formulation}\label{sec:translation_invariant_formulation}

We consider a translation-invariant envelope of the dual objective, obtained by minimizing over the translation parameter. This construction restores invariance by design and allows us to separate the essential degrees of freedom of the dual from the spurious translation direction.

We show that, under mild regularity assumptions on the cost and the underlying spaces, the envelope functional admits a compactly anchored minimization domain and is uniformly strongly convex on this set. This result is the cornerstone of our analysis: it provides a stable geometric framework that does not rely on divergence-specific properties and will be crucial for establishing continuity and statistical guarantees.

 Following \cite{sejourne2023unbalanced},  we introduce the
\emph{translation-invariant envelope} 
\begin{equation}
    \Tt^{\alpha,\beta}(f,g)
\coloneqq
\inf_{\lambda\in\mathbb R}\Kk^{\alpha,\beta}(f+\lambda,g-\lambda).
\end{equation}
By definition,
for all $\lambda\in\RR$, $\Tt^{\al,\beta}(f-\lambda,g+\lambda)=\Tt(f,g)$.

While translation invariance of the dual formulation of unbalanced optimal transport has been previously exploited for algorithmic purposes to improve the stability of Sinkhorn iterations \cite{sejourne2022translation}, our work takes a complementary perspective by treating this invariance as a genuine geometric symmetry of the dual problem. By explicitly passing to the associated quotient space, we recover strong convexity of the dual objective in its intrinsic degrees of freedom, which allows us to develop a unified stability and statistical analysis of unbalanced entropic optimal transport.

\subsection{Envelope definition and strong convexity}
We show that $\Tt$ is strongly convex with respect to the  quotient space with $\int f d\alpha = 0$. In the subsequent section, we show that this strong convexity property leads to a simple proof on sample complexity bounds for the corresponding transport plans.


\begin{thm}[Compact anchoring and strong convexity of $\Tt$] \label{thm:compact_anchor}
Assume that $c$ is $L$-Lipschitz on $\Xx\times\Yy$ and that
Assumption~\ref{ass:compact} holds.
Set $m\coloneqq m_\alpha m_\beta$.
Let $U$ be any finite constant such that $U\geq \inf_{f,g}\Kk(f,g)$. In particular,
\[
U\eqdef  m_\alpha \sup_{|t|\le \|c\|_\infty}\varphi_1^*(t)
+m_\beta \varphi_2^*(0)
+\eta\,m_\alpha m_\beta
\]
suffices. Define
\[
C_{\mathrm{FY}}\coloneqq
m_\alpha\varphi_1(m_\beta)+m_\beta\varphi_2(m_\alpha)
+m\|c\|_\infty
\]
and
\[
M \eqdef \|c\|_\infty+\frac{2(U+C_{\mathrm{FY}})}{m}
+4L\big(\diam(\Xx)+\diam(\Yy)\big).
\]
Define the anchored compact set
\[
\Pp_0^{\alpha,\beta}
\coloneqq
\Bigl\{(f,g):\;
f,g\in\Cc_L,\;
\int f\,d\alpha=0,\;
\|f\|_\infty\le M,\;
\|g\|_\infty\le M
\Bigr\},
\]
where $\Cc_L$ denotes the set of $L$-Lipschitz functions.
Then
\begin{equation}\label{eq:UOT-kantoro-explicit}
\min_{f,g}\Kk^{\alpha,\beta}(f,g)
=
\min_{(f,g)\in\Pp_0^{\alpha,\beta}}
\Tt^{\alpha,\beta}(f,g).
\end{equation}
Furthermore, with $\mu \;=\; \eta^{-1}e^{-M/\eta}$,
$\Tt^{\alpha,\beta}$ is $\mu$-strongly convex on $\Pp_0^{\alpha,\beta}$
with respect to
\[
\|(u,v)\|_{\alpha,\beta}^2
\coloneqq m_\beta \int u^2\,d\alpha+ m_\alpha \int v^2\,d\beta.
\]

\end{thm}

\begin{proof}
For brevity write $\Kk=\Kk^{\alpha,\beta}$ and $\Tt=\Tt^{\alpha,\beta}$.

\medskip
\noindent\textbf{Step 1: Explicit bound on $f\oplus g-c$.}
Let $h=f\oplus g-c$ and $q=h/\eta$.
By the Fenchel--Young inequality,
\[
\varphi_1^\ast(-f)+\varphi_1(m_\beta)+ m_\beta f\ge 0,
\qquad
\varphi_2^\ast(-g)+\varphi_2(m_\alpha)+m_\alpha g\ge 0.
\]
So,
\[
\Kk(f,g)+C_{\mathrm{FY}}
\ge
\eta\int (e^{q}-q)\,d(\alpha\otimes\beta).
\] 
Here $C_{\mathrm{FY}}$ is the constant from the statement and
$q\eqdef (f\oplus g - c)/\eta$.

Any minimizer $(f,g)$ of $\Kk$ satisfies $\Kk(f,g)\le \Kk(f_0,g_0)$,
where we define
\[
f_0(x)\coloneqq \inf_{y\in\Yy} c(x,y),
\qquad g_0(y)\coloneqq 0.
\]
Then $f_0\oplus g_0-c\le 0$, hence
$\exp((f_0\oplus g_0-c)/\eta)\le 1$ and therefore
\[
\eta\Big\langle \exp\Big(\frac{f_0\oplus g_0-c}{\eta}\Big),\alpha\otimes\beta\Big\rangle
\le \eta\,m_\alpha m_\beta.
\]
Since $-f_0(x)\in[-\|c\|_\infty,\|c\|_\infty]$, we also have
\[
\langle \varphi_1^*(-f_0),\alpha\rangle
\le m_\alpha \sup_{|t|\le \|c\|_\infty}\varphi_1^*(t).
\]
Thus
\[
\inf_{f,g}\Kk^{\alpha,\beta}(f,g)
\le \Kk^{\alpha,\beta}(f_0,g_0)
\le
m_\alpha \sup_{|t|\le \|c\|_\infty}\varphi_1^*(t)
+m_\beta \varphi_2^*(0)
+\eta\,m_\alpha m_\beta.
\]
This proves that the displayed choice of $U$ in the statement is admissible.

For any chosen upper bound $U$, we can restrict to considering $q$ such that 
\[
\int (e^{q}-q)\,d(\alpha\otimes\beta)
\le \frac{U+C_{\mathrm{FY}}}{\eta}
\eqdef D\,m.
\]
By Jensen's inequality,
$e^{\bar q}-\bar q\le D$ with
$\bar q=\frac1m\int q\,d(\alpha\otimes\beta)$,
which implies $|\bar q|\le 2D$ assuming $D\geq 1$.
Therefore
\[
\Big|\int (f\oplus g - c)\,d(\alpha\otimes\beta)\Big|
=\eta m|\bar q|
\le 2D \eta m.
\]
Since $f,g,c$ are $L$-Lipschitz, $(f\oplus g - c)$ is $3L$-Lipschitz on
$\Xx\times\Yy$, and for all $(x,y)$,
\[
|(f\oplus g - c)(x,y)|
\le
\Big|\frac{1}{m}\int (f\oplus g - c)\,d(\alpha\otimes\beta)\Big|
+3L\big(\diam(\Xx)+\diam(\Yy)\big).
\]
Thus,  we can restrict minimization of $\Kk^{\alpha,\beta}$ to $(f,g)$ that satisfies
\begin{equation}\label{eq:Mh_def}
\|f\oplus g-c\|_\infty \;\le\;
M_0 \;\coloneqq\;
\frac{2(U+C_{\mathrm{FY}})}{m}+3L\big(\diam(\Xx)+\diam(\Yy)\big).
\end{equation}

\medskip
\noindent\textbf{Step 2: Compact anchoring.}
Fix the gauge $\int f\,d\alpha=0$.
For $f\in\Cc_L$,
\[
|f(x)|
=
\Big|\frac1{m_\alpha}\int (f(x)-f(x'))\,d\alpha(x')\Big|
\le L\,\diam(\Xx),
\]
so $\|f\|_\infty\le L\diam(\Xx)$.
Using $\|f\oplus g-c\|_\infty\le M_0$ gives
\[
|g(y)|
\le \|c\|_\infty+M_0+\|f\|_\infty,
\]
which yields 
\begin{equation}\label{eq:M_def}
\|f\|_\infty \le L\,\diam(\Xx),
\qquad
\|g\|_\infty \le
\|c\|_\infty + \frac{2(U+C_{\mathrm{FY}})}{m}
+4L\big(\diam(\Xx)+\diam(\Yy)\big) \eqdef M.
\end{equation}
Since translations preserve $\Tt$,
\eqref{eq:UOT-kantoro-explicit} follows.

\medskip
\noindent\textbf{Step 3: Strong convexity.}
Decompose $\Tt=F+\Tt_0$ with
\[
F(f,g)=\eta\Big\langle e^{(f\oplus g-c)/\eta},\alpha\otimes\beta\Big\rangle,
\qquad
\Tt_0(f,g)=
\inf_{\lambda\in\mathbb R}
\bigl(
\langle \varphi_1^\ast(-f+\lambda),\alpha\rangle
+
\langle \varphi_2^\ast(-g-\lambda),\beta\rangle
\bigr).
\]
The map $\Tt_0$ is convex.
On $\Pp_0^{\alpha,\beta}$ we have
$\|f\oplus g-c\|_\infty\le M_0$, hence the Hessian of $F$
is bounded below by $\eta^{-1}e^{-M_0/\eta}$.
A direct expansion using $\int f\,d\alpha=\int f'\,d\alpha=0$ gives 
\[
\int\big((f-f')\oplus(g-g')\big)^2\,d(\alpha\otimes\beta)
=
m_\beta \|f-f'\|_{L^2(\alpha)}^2+ m_\alpha \|g-g'\|_{L^2(\beta)}^2,
\]
which yields the claimed strong convexity with $\eta^{-1} e^{-M_0/\eta}\geq \eta^{-1} e^{-M/\eta}$.
\end{proof}

\subsection{Continuity and stability}
Thanks to the strong convexity property, we have the following continuity property  of the potentials:

\begin{prop}[Continuity of quotient dual potentials and weak convergence of the primal plans]
\label{prop:convergence_primal_weak}
Assume Assumption~\ref{ass:compact}. Let $\alpha^n \rightharpoonup \alpha$ and $\beta^n \rightharpoonup \beta$ weakly.
For each $n$, let $(f_n,g_n)\in \Pp_0^{\alpha^n,\beta^n}$ be the (unique) minimizer of the translation-invariant
dual objective $\Tt^{\alpha^n,\beta^n}$. Then $(f_n,g_n)$ converges uniformly to the (unique) minimizer
$(f,g)\in \Pp_0^{\alpha,\beta}$ of $\Tt^{\alpha,\beta}$.

Moreover, defining
\[
\pi^n \coloneqq \exp\!\Big(\frac{f_n\oplus g_n-c}{\eta}\Big)\,(\alpha^n\otimes\beta^n),
\qquad
\pi \coloneqq \exp\!\Big(\frac{f\oplus g-c}{\eta}\Big)\,(\alpha\otimes\beta),
\]
we have $\pi^n \rightharpoonup \pi$ weakly; equivalently,
$\langle h,\pi^n-\pi\rangle\to 0$ for all $h\in C(\Xx\times\Yy)$.
\end{prop}

\begin{proof}
By Assumption~\ref{ass:compact}, $f_n,g_n\in \Cc_L$ with uniform $L^\infty$ bounds.
Since $\Xx$ and $\Yy$ are compact, Arzel\`a--Ascoli yields a subsequence (not relabeled) such that
$f_n\to f$ and $g_n\to g$ uniformly, for some $f,g\in \Cc_L$.

We first show that $(f,g)\in\Pp_0^{\alpha,\beta}$.
Since $\int f_n\,d\alpha^n=0$ for all $n$,
\[
0=\int f_n\,d\alpha^n
= \int f\,d\alpha
+ \int (f_n-f)\,d\alpha^n
+ \int f\,d(\alpha^n-\alpha).
\]
The second term tends to $0$ by uniform convergence, and the third by weak convergence
$\alpha^n\rightharpoonup\alpha$ and continuity of $f$. Hence $\int f\,d\alpha=0$.

Next, we identify $(f,g)$ as the minimizer of $\Tt^{\alpha,\beta}$.
Let $(\hat f,\hat g)\in\Pp_0^{\alpha,\beta}$ be arbitrary and define
$\kappa_n\coloneqq \frac{1}{m_\alpha} \int \hat f\,d\alpha^n$, 
$\hat f_n\coloneqq \hat f-\kappa_n$,
$\hat g_n\coloneqq \hat g+\kappa_n$.
Then $(\hat f_n,\hat g_n)\in\Pp_0^{\alpha^n,\beta^n}$ and
$\hat f_n\oplus \hat g_n=\hat f\oplus \hat g$.
By optimality of $(f_n,g_n)$,
\[
\Tt^{\alpha^n,\beta^n}(f_n,g_n)\le \Tt^{\alpha^n,\beta^n}(\hat f_n,\hat g_n).
\]
Since $f_n\to f$, $g_n\to g$ uniformly, $c$ is Lipschitz, and
$\alpha^n\otimes\beta^n\rightharpoonup\alpha\otimes\beta$,
both sides converge to $\Tt^{\alpha,\beta}(f,g)$ and $\Tt^{\alpha,\beta}(\hat f,\hat g)$ respectively.
Thus $\Tt^{\alpha,\beta}(f,g)\le \Tt^{\alpha,\beta}(\hat f,\hat g)$, proving optimality.
By strong convexity of $\Tt^{\alpha,\beta}$ on $\Pp_0^{\alpha,\beta}$, the minimizer is unique,
so the full sequence converges uniformly.

Finally, since $f_n\oplus g_n\to f\oplus g$ uniformly and
$\alpha^n\otimes\beta^n\rightharpoonup\alpha\otimes\beta$,
we have for all $h\in C(\Xx\times\Yy)$,
\[
\langle h,\pi^n\rangle
= \int h\,e^{(f_n\oplus g_n-c)/\eta}\,d(\alpha^n\otimes\beta^n)
\;\longrightarrow\;
\int h\,e^{(f\oplus g-c)/\eta}\,d(\alpha\otimes\beta)
= \langle h,\pi\rangle.
\]
\end{proof}

\begin{rem}
The proposition above is similar to Proposition 10 of \cite{sejourne2019sinkhorn}, however, strong convexity of the translation invariant operator and hence our continuity result  hold under much weaker assumptions: we only need the cost $c$ to be  Lipschitz  and compact domains.
\end{rem}

\subsection{Uniform compactness of the translation parameter}\label{subsec:uniform_translation_compactness}

Although the envelope formulation removes the gauge freedom, it still involves an optimization over an a priori unbounded translation parameter $\lambda$. For subsequent stability and sample complexity arguments, it is essential to control this remaining degree of freedom uniformly.

In this section, we show that under a mild domain condition on the marginal divergences, the translation parameter can in fact be restricted to a compact interval, uniformly over all admissible dual potentials. This condition is satisfied by most $\varphi$-divergences used in unbalanced OT and excludes only the balanced case, where translation invariance is exact and no restriction is needed.

Our result yields explicit bounds on the translation parameter and further implies uniform 
 bounds on the dual potentials themselves. These compactness properties play a key role in controlling empirical processes in the statistical analysis that follows.
\begin{assumption}[Bounded translations]\label{ass:bounded}
There exist $a,a'\in\mathrm{dom}(\varphi_1)$ and $b,b'\in\mathrm{dom}(\varphi_2)$ such that
\[
bm_\beta > am_\alpha,
\qquad
b'm_\beta < a'm_\alpha.
\]
\end{assumption}

\begin{rem}
    Assumption~\ref{ass:bounded} is a mild \emph{domain condition}: it asks for two pairs
$(a,b)$ and $(a',b')$ with $a,a'\in\mathrm{dom}(\varphi_1)$ and $b,b'\in\mathrm{dom}(\varphi_2)$ such that
$bm_\beta-am_\alpha>0$ and $b'm_\beta-a'm_\alpha<0$.
In particular, if $\mathrm{dom}(\varphi_1)$ and $\mathrm{dom}(\varphi_2)$ both contain an interval
(e.g.\ $(0,\infty)$, or $[0,\infty)$), then the assumption always holds: pick any $a\neq a'$ in
$\mathrm{dom}(\varphi_1)$ and choose $b>b':=\frac{m_\alpha}{m_\beta}a'$ in $\mathrm{dom}(\varphi_2)$
(and similarly $b<\frac{m_\alpha}{m_\beta}a$ for the other inequality). This covers the common
$\varphi$-divergences used in unbalanced OT, whose generators are finite on $(0,\infty)$, such as
\emph{KL} ($\varphi(p)=p\log p-p+1$), \emph{$\chi^2$} ($\varphi(p)=(p-1)^2$), \emph{Hellinger}
($\varphi(p)=(\sqrt p-1)^2$), \emph{Jensen--Shannon} (an $f$-divergence with finite generator on $(0,\infty)$),
and more generally the \emph{$\alpha$-divergences} (power divergences), all of which are standard $f$-divergences
with $\mathrm{dom}(\varphi)=(0,\infty)$ (or containing $(0,\infty)$). 
In contrast, for \emph{balanced OT} the marginal penalties are hard constraints,
corresponding to the indicator divergence $\iota_{\{1\}}$ (i.e.\ $\pi_1=\alpha$ and $\pi_2=\beta$); then the dual
has the usual gauge invariance and the envelope minimization is trivial, effectively $R=0$ (no translation search).

\end{rem}

The next proposition shows that under Assumption~\ref{ass:bounded}, one can further restrict
the optimization over $\lambda$ in $\Tt$ to a compact interval.
Our proof relies on the Fenchel--Young inequality. While related compactness results appear in
\cite{sejourne2023unbalanced}, the proof strategies there depend on the particular divergence:
one argument applies when $\varphi^\ast$ is strictly convex and admits subgradients converging to
$0$ or $\infty$ \cite[Lemma~5]{sejourne2023unbalanced}, while the TV and range divergences are handled
separately. In contrast, Assumption~\ref{ass:bounded}  covers all these cases.

\begin{thm}[Restriction to bounded translations with explicit constants]\label{thm:restrictCompactSets}
Assume Assumptions~\ref{ass:compact} and~\ref{ass:bounded}.
Let $\Pp_0^{\alpha,\beta}$ be the anchored compact set from
Theorem~\ref{thm:compact_anchor}, and let $M, U>0$ be the 
constants there.
There exists $R$ such that 
\begin{equation}\label{eq:restrict_lambda_interval}
\inf_{f,g\in\Pp_0^{\al,\beta}}\Tt^{\alpha,\beta}(f,g)
= \inf_{f,g\in\Pp_0^{\al,\beta}}
\min_{\lambda\in[-R,R]}\Kk^{\alpha,\beta}(f+\lambda,g-\lambda).
\end{equation}
Consequently, one can restrict the minimization of $\Kk^{\alpha,\beta}$ to $(f,g)$ such that 
\[
\|f\|_\infty\le M+R,
\qquad
\|g\|_\infty\le M+R.
\]

Concretely, we can 
define 
\[
s_+ \coloneqq bm_\beta-am_\alpha>0,
\qquad
s_- \coloneqq a'm_\alpha-b'm_\beta>0,
\]
and the constants
\begin{align*}
C_{a,b}(M)
&\coloneqq
|a|\,M\,m_\alpha+\varphi_1(a)\,m_\alpha
+|b|\,M\,m_\beta+\varphi_2(b)\,m_\beta,\\
C_{a',b'}(M)
&\coloneqq
|a'|\,M\,m_\alpha+\varphi_1(a')\,m_\alpha
+|b'|\,M\,m_\beta+\varphi_2(b')\,m_\beta.
\end{align*}
Then the quantity
\begin{equation}\label{eq:R_explicit}
R \;\coloneqq\;
\max\Big\{
\frac{U+C_{a,b}(M)}{s_+},
\frac{U+C_{a',b'}(M)}{s_-}
\Big\}
\end{equation}
\end{thm}

\begin{proof}
Fix $(f,g)$ and abbreviate $\Kk=\Kk^{\alpha,\beta}$.
Starting from the Fenchel--Young inequality on $\varphi_1^*$, $\varphi_2^*$, for any $\pi_1\ll\alpha$ and $\pi_2\ll\beta$,
\begin{equation}\label{eq:K_lambda_lb_rewrite}
\begin{split}
\Kk(f+\lambda,g-\lambda)\ge\;
&-\langle f,\pi_1\rangle-\langle g,\pi_2\rangle
+\eta\Big\langle\exp\!\Big(\frac{f\oplus g-c}{\eta}\Big),\alpha\otimes\beta\Big\rangle\\
&\quad
+\lambda\bigl(m_{\pi_2}-m_{\pi_1}\bigr)
-\Big\langle \varphi_1\!\Big(\frac{d\pi_1}{d\alpha}\Big),\alpha\Big\rangle
-\Big\langle \varphi_2\!\Big(\frac{d\pi_2}{d\beta}\Big),\beta\Big\rangle,
\end{split}
\end{equation}
where $m_{\pi_i}$ denotes the total mass of $\pi_i$.

\medskip
\noindent\textbf{Step 1: One-sided linear coercivity with explicit intercepts.}
Choose $\pi_1=a\alpha$ and $\pi_2=b\beta$ with $a\in\mathrm{dom}(\varphi_1)$ and
$b\in\mathrm{dom}(\varphi_2)$. Then $m_{\pi_1}=am_\alpha$, $m_{\pi_2}=bm_\beta$,
and \eqref{eq:K_lambda_lb_rewrite} gives
\begin{equation}\label{eq:K_lambda_lb_ab_rewrite}
\begin{split}
\Kk(f+\lambda,g-\lambda)\ge\;
&s_+\,\lambda
-\langle af,\alpha\rangle-\varphi_1(a)m_\alpha
-\langle bg,\beta\rangle-\varphi_2(b)m_\beta\\
&\quad
+\eta\Big\langle\exp\!\Big(\frac{f\oplus g-c}{\eta}\Big),\alpha\otimes\beta\Big\rangle .
\end{split}
\end{equation}
Since the exponential term is nonnegative, we may drop it to obtain a simpler bound.
If additionally $(f,g)\in\Pp_0^{\alpha,\beta}$, then $\|f\|_\infty,\|g\|_\infty\le M$, hence
\[
-\langle af,\alpha\rangle \ge -|a|\int |f|\,d\alpha \ge -|a|M m_\alpha,
\qquad
-\langle bg,\beta\rangle \ge -|b|M m_\beta,
\]
and therefore from \eqref{eq:K_lambda_lb_ab_rewrite},
\begin{equation}\label{eq:K_lambda_lb_ab_simplified}
\Kk(f+\lambda,g-\lambda)\ge s_+\,\lambda - C_{a,b}(M),
\qquad \forall\,\lambda\in\RR.
\end{equation}

Similarly, choosing $\pi_1=a'\alpha$ and $\pi_2=b'\beta$ yields
\begin{equation}\label{eq:K_lambda_lb_a'b'_simplified}
\Kk(f+\lambda,g-\lambda)\ge -s_-\,\lambda - C_{a',b'}(M),
\qquad \forall\,\lambda\in\RR,
\end{equation}
again for all $(f,g)\in\Pp_0^{\alpha,\beta}$.

\medskip
\noindent\textbf{Step 2: Uniform restriction of minimizers of $(\lambda,f,g)\mapsto \Kk(f+\lambda,g-\lambda)$.}
For any minimizing $(\lambda, f,g)$ we have
\[
\Kk(f+\lambda,g-\lambda)\le U,
\]
with $U$ as in Theorem \ref{thm:compact_anchor}.
Consequently,  $\Tt(f,g)\le U$ for all $(f,g)\in\Pp_0^{\alpha,\beta}$.

Now if $\lambda\ge 0$, by
\eqref{eq:K_lambda_lb_ab_simplified},
\[
U\geq \Kk(f+\lambda,g-\lambda)\ge s_+\lambda-C_{a,b}(M),
\]
so no minimizer over $\lambda\in\RR$ can lie beyond $(U+C_{a,b}(M))/s_+$ on the positive side.
Similarly, by
\eqref{eq:K_lambda_lb_a'b'_simplified},
\[
U\geq \Kk(f+\lambda,g-\lambda)\ge -s_-\lambda-C_{a',b'}(M),
\]
so no minimizer can lie beyond $-(U+C_{a',b'}(M))/s_-$ on the negative side.
Therefore \eqref{eq:restrict_lambda_interval} holds with
\[
R \coloneqq
\max\Big\{
\frac{U+C_{a,b}(M)}{s_+},
\frac{U+C_{a',b'}(M)}{s_-}
\Big\}.
\]

\end{proof}

\section{Sample complexity}\label{sec:sample_complexity}

We now turn the deterministic geometry of Section~\ref{sec:translation_invariant_formulation} into a statistical guarantee. The object of interest is the transport plan itself: given empirical marginals $\alpha^n,\beta^n$, we want to control the deviation between the empirical UOT coupling and the population coupling through bounded test functions. The argument is organized in two steps. Subsection~\ref{subsec:translation_minimizer_stability} isolates the one-dimensional stability mechanism for the scalar translation parameter. Subsection~\ref{subsec:plan_sample_complexity} then combines this mechanism with the strong convexity of the anchored envelope and the concentration estimates of Section~\ref{sec:concentration_bounds} to prove the plan-level sample-complexity theorem.

The deterministic input for the section is the compact anchoring and translation control obtained above. Under Assumptions~\ref{ass:compact} and~\ref{ass:bounded}, Theorem~\ref{thm:restrictCompactSets} gives a deterministic interval $[-R,R]$ such that
\begin{equation}\label{eq:trans_inv_compact}
\inf_{f,g\in\Pp_0^{\al,\beta}}\Tt^{\alpha,\beta}(f,g)
= \inf_{f,g\in\Pp_0^{\al,\beta}}
\min_{\lambda\in[-R,R]}\Kk^{\alpha,\beta}(f+\lambda,g-\lambda).
\end{equation}
In the balanced case the gauge can be fixed directly and we take $R=0$. The additional curvature and subgradient regularity assumptions needed for the final theorem are therefore introduced only in Subsection~\ref{subsec:plan_sample_complexity}; the preliminary proposition below only needs local conditions on the one-dimensional scalar problem.

\subsection{Stability of the translation minimizer}\label{subsec:translation_minimizer_stability}

The translation-invariant envelope separates the intrinsic potentials from the scalar shift $\lambda$. Once the anchored potentials are fixed, the only remaining scalar degree of freedom is the minimizer of the function $\lambda\mapsto \Kk(f+\lambda,g-\lambda)$. The empirical objective may choose a different minimizer, and the next proposition shows that this displacement is small whenever the empirical scalar objective has curvature. This one-dimensional estimate is the bridge between the deterministic translation control of Theorem~\ref{thm:restrictCompactSets} and the comparison of empirical and population dual certificates.

\begin{prop}[Stability of the translation minimizer]\label{prop:lambda_stability}
Fix bounded $f,g$ with $\|f\|_\infty,\|g\|_\infty\le M$ and define
\[
G(\lambda)\coloneqq \Kk^{\alpha,\beta}(f+\lambda,g-\lambda),
\qquad
G_n(\lambda)\coloneqq \Kk^{\alpha^n,\beta^n}(f+\lambda,g-\lambda),
\quad \lambda\in[-R,R].
\]
Assume that $G_n$ is $\gamma$--strongly convex on $[-R,R]$, and let
$$\lambda^\star\in\arg\min_{[-R,R]} G(\lambda)\qandq \lambda_n^\star\in\arg\min_{[-R,R]} G_n(\lambda).$$
Suppose moreover that $\lambda^\star$ admits measurable stationarity selections
\[
\psi_1^\star(\cdot)\in \partial\varphi_1^*(-f(\cdot)-\lambda^\star),
\qquad
\psi_2^\star(\cdot)\in \partial\varphi_2^*(-g(\cdot)+\lambda^\star),
\]
satisfying
\[
-\int \psi_1^\star\,d\alpha+\int \psi_2^\star\,d\beta=0,
\qquad
|\psi_1^\star|\le B_1,\quad |\psi_2^\star|\le B_2 .
\]
Then, for all $t>0$, with probability at least $1-e^{-t}$,
\[
|\lambda_n^\star-\lambda^\star|
\;\lesssim\;
\frac{m_\alpha B_1+m_\beta B_2}{\gamma}\,\sqrt{\frac{t}{n}}.
\]
\end{prop}
\begin{proof}
Since the entropic term is invariant under $(f,g)\mapsto(f+\lambda,g-\lambda)$,
a subgradient of $G$ at $\lambda$ is given by
\[
\partial G(\lambda)\ni -\int \psi_1\,d\alpha + \int \psi_2\,d\beta,
\quad
\psi_1(\cdot)\in \partial\varphi_1^*(-f(\cdot)-\lambda),\ 
\psi_2(\cdot)\in \partial\varphi_2^*(-g(\cdot)+\lambda).
\]
Using the stationarity selections at $\lambda^\star$, the corresponding empirical subgradient satisfies
\[
s_n \coloneqq -\int \psi_1^\star\,d(\alpha^n-\alpha)+\int \psi_2^\star\,d(\beta^n-\beta)
\in \partial G_n(\lambda^\star).
\]
By $\gamma$--strong convexity of $G_n$ on $[-R,R]$,
\[
|\lambda_n^\star-\lambda^\star|
\le \frac{1}{\gamma}\,\mathrm{dist}\!\bigl(0,\partial G_n(\lambda^\star)\bigr)
\le \frac{1}{\gamma}|s_n|.
\]
Hoeffding's inequality and the bounds
$|\psi_1^\star|\le B_1$, $|\psi_2^\star|\le B_2$ yield the stated rate.

\end{proof}

\subsection{Plan-level sample complexity}\label{subsec:plan_sample_complexity}

We now impose the uniform assumptions that make the preceding scalar estimate usable along the empirical and population minimizers. The first condition bounds the marginal subgradients on the compact range of admissible translated potentials. The second supplies curvature in the translation direction, preventing an almost-flat scalar mode from amplifying sampling noise. The third transfers marginal subgradient certificates from the population translation to the empirical one. These requirements are automatic in the balanced gauge-fixed case, and they hold for the standard smooth strictly convex unbalanced divergences on compact potential ranges.

\begin{assumption}[Translation curvature and subgradient regularity]\label{ass:strong_conv}
Let $R\in [0,\infty)$ be the translation radius from Theorem~\ref{thm:restrictCompactSets}
in the unbalanced case, and set $R\coloneqq 0$ in the balanced case.
Let $M>0$ be the uniform $L^\infty$ bound on dual potentials from Theorem~\ref{thm:restrictCompactSets},
and define
\[
I_{M,R}\coloneqq [-(M+R),\, M+R].
\]

Assume the following conditions.
\begin{enumerate}
\item[(i)] \textbf{Bounded subgradients on the relevant range.}
For $i=1,2$,
\[
B_i \coloneqq \sup_{u\in I_{M,R}}\,\sup_{v\in \partial \varphi_i^*(u)} |v| \;<\; \infty.
\]

\item[(ii)] \textbf{Strong convexity in the translation direction.}
At least one of $\varphi_1^*$ or $\varphi_2^*$ is strongly convex on compact sets. If we can set $R=0$ in \eqref{eq:restrict_lambda_interval}, then this condition is not required.

\item[(iii)] \textbf{Local Lipschitz property of marginal subgradients.}
For each $i\in\{1,2\}$, there exists $L_i<\infty$ such that, for all $p,q\in I_{M,R}$ with nonempty $\partial\varphi_i^*(p)$ and $\partial\varphi_i^*(q)$, every
$\psi\in \partial\varphi_i^*(p)$ admits a companion $\tilde\psi\in\partial\varphi_i^*(q)$ satisfying
\[
|\tilde\psi-\psi|\le L_i|p-q|.
\]
\end{enumerate}
\end{assumption}

\begin{rem*}
Assumption~\ref{ass:strong_conv}(iii) holds if $\varphi_i^*\in C^{1,1}$ on a neighborhood of $I_{M,R}$
(with $L_i=\mathrm{Lip}(\nabla\varphi_i^*)$), and it also holds with $L_i=0$ whenever
$\partial\varphi_i^*(u)$ is a singleton independent of $u$ on $I_{M,R}$ (e.g.\ affine/indicator-type conjugates arising from TV penalties on the relevant range).
\end{rem*}

\begin{rem*}
In the balanced OT setting, the dual objective is invariant under $(f,g)\mapsto(f+\lambda,g-\lambda)$,
and we fix the gauge (e.g.\ $\int f\,d\alpha=0$). Equivalently, the envelope optimization has $R=0$,
so Assumption~\ref{ass:strong_conv} is satisfied trivially.
Assumption~\ref{ass:strong_conv} (ii) holds for KL and, more broadly,
for smooth strictly convex $f$-divergences (e.g.\ $\chi^2$, Hellinger, Jensen--Shannon, and $\alpha$-divergences):
on any compact interval, their conjugates have a strictly positive second derivative and hence a positive strong convexity
modulus. 
\end{rem*}

Under Assumption~\ref{ass:strong_conv}, the empirical envelope is stable in the anchored variables and the remaining translation drift is controlled by Proposition~\ref{prop:lambda_stability}. The theorem below converts these dual estimates into a high-probability bound for the optimal coupling through the exponential formula \eqref{RadonDensity}.

\begin{thm}[Sample complexity for unbalanced entropic OT]\label{SampleComplextiyuOTTHeorem}
Assume Assumptions~\ref{ass:compact}, \ref{ass:bounded}, and~\ref{ass:strong_conv}.
Let $(z_i,w_i)_{i=1}^n \overset{iid}{\sim} \xi$, where $\xi$ has marginals
$\xi_1=\alpha/m_\alpha$ and $\xi_2=\beta/m_\beta$.
Let $\pi^\star$ denote the UOT solution of \eqref{eq:UOT2}, and let $\pi^n$ denote the UOT solution
with $(\alpha,\beta)$ replaced by the empirical measures
\[
\alpha^n := \frac{m_\alpha}{n}\sum_{i=1}^n \delta_{z_i},
\qquad
\beta^n := \frac{m_\beta}{n}\sum_{j=1}^n \delta_{w_j}.
\]

Then for all $t>0$ and all $h \in C_c(X\times Y)$, with probability at least $1-e^{-t}$,
\begin{equation}\label{eq:samp-comp-forward}
\big| \langle h, \pi^\star-\pi^n\rangle \big|
\;\le\;
C_1\,e^{C_2(M+\|c\|_\infty)/\eta}\,
\|h\|_{L^\infty(\alpha\otimes \beta)}
\sqrt{\frac{m_\alpha m_\beta(\log n+t)}{n}},
\end{equation}
where $C_1>0$ depends only on $m_\alpha,m_\beta,\eta$, and the constants in
Assumptions~\ref{ass:bounded}--\ref{ass:strong_conv}, and $C_2>0$ is a numerical constant.

If $\xi=\frac{1}{m_\alpha m_\beta}\alpha\otimes\beta$, the $\log n$ factor can be removed.
\end{thm}

\begin{proof}[Proof of Theorem \ref{SampleComplextiyuOTTHeorem}]
Let
\[
p_\infty^\star := \frac{d\pi^\star}{d(\alpha\otimes\beta)},
\qquad
p_n^\star := \frac{d\pi^n}{d(\alpha^n\otimes\beta^n)}.
\]
Let $(f_\infty^\star,g_\infty^\star)\in\Pp_0^{\alpha,\beta}$ be a minimizer of the population envelope
$\Tt^{\alpha,\beta}$, and let $(f_n^\star,g_n^\star)$ be an optimal dual pair for $\Kk_n$.
By Theorem~\ref{thm:restrictCompactSets}, all these potentials are uniformly bounded in $L^\infty$:
there exists $M>0$  such that
\begin{equation}\label{eq:sup_bounds_all}
\|f_\infty^\star\|_\infty,\|g_\infty^\star\|_\infty,\|f_n^\star\|_\infty,\|g_n^\star\|_\infty \le M.
\end{equation}
In particular,
\begin{equation}\label{eq:p_infty_sup_bd}
\|p_\infty^\star\|_{L^\infty(\alpha\otimes\beta)}
\le \exp\!\Big(\frac{2M+\|c\|_\infty}{\eta}\Big).
\end{equation}

\medskip

By the triangle inequality,
\begin{align}
\big| \langle h, \pi^\star-\pi^n\rangle \big|
&\le
\big|\langle h p_\infty^\star, \alpha\otimes\beta-\alpha^n\otimes\beta^n\rangle\big|
+\|h\|_\infty
\big|\langle p_\infty^\star-p_n^\star,\alpha^n\otimes\beta^n\rangle\big|.
\label{eq:tri_decomp_final}
\end{align}

To control the first term on the RHS of \eqref{eq:tri_decomp_final}:
Apply Lemma~\ref{UstatisticsLemma} to
$\tilde h = hp_\infty^\star-\langle hp_\infty^\star,\alpha\otimes\beta\rangle/(m_\alpha m_\beta)$. 
Using $\|\tilde h\|_\infty\le 2\|h\|_\infty\|p_\infty^\star\|_\infty$ and \eqref{eq:p_infty_sup_bd}, we obtain:
for all $t>0$, with probability at least $1-e^{-t}$,
\begin{equation}\label{eq:term1_conc}
\big|\langle h p_\infty^\star, \alpha\otimes\beta-\alpha^n\otimes\beta^n\rangle\big|
\;\lesssim\;
\|h\|_\infty\,\exp\!\Big(\frac{2M+\|c\|_\infty}{\eta}\Big)\, m_\alpha m_\beta\;
\sqrt{\frac{t}{n}}.
\end{equation}

\medskip
It remains to bound the second term on the RHS of \eqref{eq:tri_decomp_final}.

\medskip
\textit{Lipschitz control of the primal densities.}
On the interval $[-(2M+\|c\|_\infty),\,2M+\|c\|_\infty]$, the map $u\mapsto e^{u/\eta}$ is
$\eta^{-1}\exp((2M+\|c\|_\infty)/\eta)$--Lipschitz. Hence, for all $(x,y)$,
\[
|p_\infty^\star(x,y)-p_n^\star(x,y)|
\le
\eta^{-1}\exp\!\Big(\frac{2M+\|c\|_\infty}{\eta}\Big)
\Big(|f_\infty^\star(x)-f_n^\star(x)|+|g_\infty^\star(y)-g_n^\star(y)|\Big).
\]
Integrating against $\alpha^n\otimes\beta^n$ and using Cauchy--Schwarz yields
\begin{equation}\label{eq:density_reduce_final}
\big|\langle p_\infty^\star-p_n^\star,\alpha^n\otimes\beta^n\rangle\big|
\;\lesssim\;
\eta^{-1}\exp\!\Big(\frac{2M+\|c\|_\infty}{\eta}\Big)\,
\|(f_\infty^\star-f_n^\star,g_\infty^\star-g_n^\star)\|_{\alpha^n,\beta^n}.
\end{equation}

\medskip

\noindent\emph{Empirical centering of the population potentials.}
Define the empirical centering constant and centered potentials 
\begin{equation}\label{eq:center_pop_alpha_n}
\kappa_n \coloneqq \frac{1}{m_\alpha} \int f_\infty^\star\,d\alpha^n,
\qquad
\bar f_\infty^\star \coloneqq f_\infty^\star-\kappa_n,
\qquad
\bar g_\infty^\star \coloneqq g_\infty^\star+\kappa_n .
\end{equation}
Then $\int \bar f_\infty^\star\,d\alpha^n=0$ and
\[
\|\bar f_\infty^\star\|_\infty\le 2M,\qquad \|\bar g_\infty^\star\|_\infty\le 2M,
\qquad
(\bar f_\infty^\star\oplus\bar g_\infty^\star)=(f_\infty^\star\oplus g_\infty^\star).
\]
In particular, the density $p_\infty^\star$ remains the same when expressed via
$(\bar f_\infty^\star,\bar g_\infty^\star)$. So, we can replace $(f^\star_\infty, g^\star_\infty)$ in \eqref{eq:density_reduce_final} by their centred counterparts $(\bar f_\infty^\star,\bar g_\infty^\star)$

Moreover, since $\int f_\infty^\star\,d\alpha=0$ (population anchoring), we have
$\kappa_n=\int f_\infty^\star\,d(\alpha^n-\alpha)$, and Hoeffding's inequality gives: for all $t>0$,
with probability at least $1-e^{-t}$,
\begin{equation}\label{eq:kappa_conc}
|\kappa_n|
\;\lesssim\;
M m_\alpha\sqrt{\frac{t}{n}}.
\end{equation}

\smallskip
\noindent\emph{Empirical envelope and its minimizers.}
Define
\[
\Tt_n(f,g)\coloneqq \min_{\lambda\in[-R,R]} \Kk_n(f+\lambda,g-\lambda),
\qquad (f,g)\in\Pp_0^{\alpha^n,\beta^n},
\]
where $R$ is the deterministic translation bound from Theorem~\ref{thm:restrictCompactSets}.

By Theorem~\ref{thm:compact_anchor} (applied with $\alpha^n,\beta^n$) the envelope $\Tt_n$ is $\mu$--strongly convex on
$\Pp_0^{\alpha^n,\beta^n}$ w.r.t.\ $\|\cdot\|_{\alpha^n,\beta^n}$ with
\[
\mu=\eta^{-1}e^{-M/\eta}.
\]
Therefore, since $( f_n^\star, g_n^\star)$ minimize  $\Tt_n$,
\begin{equation}\label{eq:PL_Tn_patched}
\|( f_n^\star-\bar f_\infty^\star, g_n^\star -\bar g_\infty^\star)\|_{\alpha^n,\beta^n}
\ \le\
\frac{1}{\mu}\,
\operatorname{dist}\bigl(0,\partial \Tt_n(\bar f_\infty^\star,\bar g_\infty^\star)\bigr).
\end{equation}

\noindent\emph{Danskin--Valadier: reducing $\partial \Tt_n$ to $\partial \Kk_n$ at a minimizing shift.}
Set $\Phi_\lambda(f,g)\coloneqq \Kk_n(f+\lambda,g-\lambda)$, so $\Tt_n(f,g)=\min_{\lambda\in[-R,R]}\Phi_\lambda(f,g)$.
For each fixed $\lambda$, $\Phi_\lambda$ is convex in $(f,g)$, and $[-R,R]$ is compact, hence Danskin--Valadier gives
\[
\partial \Tt_n(f,g)
=
\operatorname{co}\Big(\bigcup_{\lambda\in\Lambda_n(f,g)} \partial \Phi_\lambda(f,g)\Big),
\qquad
\Lambda_n(f,g)\coloneqq \arg\min_{\lambda\in[-R,R]} \Phi_\lambda(f,g).
\]
Therefore for any $\lambda\in\Lambda_n(f,g)$,
\[
\operatorname{dist}(0,\partial \Tt_n(f,g))
\le
\operatorname{dist}(0,\partial \Phi_\lambda(f,g)).
\]
Moreover, by the chain rule for subdifferentials,
\begin{equation}\label{eq:Phi_subdiff_equals_K_subdiff}
\partial \Phi_\lambda(f,g) = \partial_{(f,g)}\Kk_n(f+\lambda,g-\lambda).
\end{equation}

Now pick
\[
\lambda_{n,\infty}\in\Lambda_n(\bar f_\infty^\star,\bar g_\infty^\star)
=
\arg\min_{\lambda\in[-R,R]}\Kk_n(\bar f_\infty^\star+\lambda,\bar g_\infty^\star-\lambda),
\]
and obtain
\begin{equation}\label{eq:dist_Tn_to_dist_Kn_at_lambda_n}
\operatorname{dist}\bigl(0,\partial \Tt_n(\bar f_\infty^\star,\bar g_\infty^\star)\bigr)
\le
\operatorname{dist}\Bigl(0,\partial_{(f,g)}\Kk_n(\bar f_\infty^\star+\lambda_{n,\infty},\bar g_\infty^\star-\lambda_{n,\infty})\Bigr).
\end{equation}

Given any measurable selections
\[
\psi_{1,n}\in -\partial\varphi_1^*\!\bigl(-(\bar f_\infty^\star+\lambda_{n,\infty})\bigr),
\qquad
\psi_{2,n}\in -\partial\varphi_2^*\!\bigl(-(\bar g_\infty^\star-\lambda_{n,\infty})\bigr),
\]
define
\begin{align}
q_{1,n}(\cdot)
&\coloneqq
\psi_{1,n}(\cdot)+\int p_\infty^\star(\cdot,y)\,d\beta^n(y),
\label{eq:q1_random}\\
q_{2,n}(\cdot)
&\coloneqq
\psi_{2,n}(\cdot)+\int p_\infty^\star(x,\cdot)\,d\alpha^n(x).
\label{eq:q2_random}
\end{align}
By direct inspection of $\Kk_n$ and \eqref{eq:Phi_subdiff_equals_K_subdiff}, we have
\[
(q_{1,n},q_{2,n})\in
\partial_{(f,g)}\Kk_n(\bar f_\infty^\star+\lambda_{n,\infty},\bar g_\infty^\star-\lambda_{n,\infty}),
\]
hence
\begin{equation}\label{eq:dist_by_choice}
\operatorname{dist}\Bigl(0,\partial_{(f,g)}\Kk_n(\bar f_\infty^\star+\lambda_{n,\infty},\bar g_\infty^\star-\lambda_{n,\infty})\Bigr)
\le
\|(q_{1,n},q_{2,n})\|_{\alpha^n,\beta^n}.
\end{equation}

It remains to bound $\|(q_{1,n},q_{2,n})\|_{\alpha^n,\beta^n}$ for some appropriate choice of $\psi_{1,n}$ and $\psi_{2,n}$.
Let $\lambda_\infty\in[-R,R]$ be a population minimizer of the translation problem:
\[
\lambda_\infty\in\argmin_{\lambda\in[-R,R]} \Kk(f_\infty^\star+\lambda,g_\infty^\star-\lambda).
\]
Pick corresponding KKT selections
\[
\psi_1^\star \in -\partial\varphi_1^\ast(-f_\infty^\star-\lambda_\infty),
\qquad
\psi_2^\star \in -\partial\varphi_2^\ast(-g_\infty^\star+\lambda_\infty),
\]
satisfying
\begin{align}
\psi_1^\star(\cdot)+\int p_\infty^\star(\cdot,y)\,d\beta(y)=0\quad \alpha\text{-a.e.},\label{eq:stat_f}\\
\psi_2^\star(\cdot)+\int p_\infty^\star(x,\cdot)\,d\alpha(x)=0\quad \beta\text{-a.e.}\label{eq:stat_g}.
\end{align}
Define the residuals
\begin{align}
\tilde q_{1}(\cdot)
&\coloneqq
\psi_1^\star(\cdot)+\int p_\infty^\star(\cdot,y)\,d\beta^n(y)
=
\int p_\infty^\star(\cdot,y)\,d(\beta^n-\beta)(y),
\label{eq:q1_det}\\
\tilde q_{2}(\cdot)
&\coloneqq
\psi_2^\star(\cdot)+\int p_\infty^\star(x,\cdot)\,d\alpha^n(x)
=
\int p_\infty^\star(x,\cdot)\,d(\alpha^n-\alpha)(x).
\label{eq:q2_det}
\end{align}
By Proposition~\ref{ConcentrationOfGradientProp} (applied to the bounded kernel $p_\infty^\star$), for all $t>0$,
with probability at least $1-e^{-t}$,
\begin{equation}\label{eq:q_det_conc}
\| \tilde q_{1}\|_{L^2(\alpha^n)}+\|\tilde q_{2}\|_{L^2(\beta^n)}
\;\lesssim\;
\exp\!\Big(\frac{2M+\|c\|_\infty}{\eta}\Big)
\sqrt{\frac{m_\alpha m_\beta(\log n+t)}{n}}.
\end{equation}

Finally we compare $(q_{1,n},q_{2,n})$ to $(\tilde q_1,\tilde q_2)$.
This is the only place where Assumption~\ref{ass:strong_conv} is used.
Indeed, Assumption~\ref{ass:strong_conv} (ii) 
implies that the scalar map
\[
G_n:\lambda\mapsto \Kk_n( f_\infty^\star+\lambda, g_\infty^\star-\lambda)
\]
is $\gamma$--strongly convex on $[-R,R]$ for a deterministic $\gamma>0$, depending only on the local curvature properties of $\phi_1^*$ and $\phi_2^*$. It has $\lambda_{n,\infty}+\kappa_n$ as a minimizer. By \eqref{eq:kappa_conc} and Proposition~\ref{prop:lambda_stability}, with probability at least $1-e^{-t}$,
\begin{equation}\label{eq:lambda_rate}
|\lambda_{n,\infty}-\lambda_\infty|
\;\lesssim\;
\exp\!\Big(\frac{2M+\|c\|_\infty}{\eta}\Big)
\sqrt{\frac{t}{n}}.
\end{equation}

By Assumption~\ref{ass:strong_conv}(iii), for each $i\in\{1,2\}$ we can choose measurable
$\psi_{i,n}$ and $\psi_i^\star$ so that
\[
\psi_{1,n}(x)\in -\partial\varphi_1^*\!\bigl(-(\bar f_\infty^\star(x)+\lambda_{n,\infty})\bigr),
\quad
\psi_1^\star(x)\in -\partial\varphi_1^*\!\bigl(-(f_\infty^\star(x)+\lambda_\infty)\bigr),
\]
(and similarly for $i=2$), and such that pointwise
\[
|\psi_{i,n}-\psi_i^\star|\le L_i\bigl(|\lambda_{n,\infty}-\lambda_\infty|+|\kappa_n|\bigr).
\]
Consequently,
\[
\|\psi_{1,n}-\psi_1^\star\|_{L^2(\alpha^n)}+\|\psi_{2,n}-\psi_2^\star\|_{L^2(\beta^n)}
\;\lesssim\;
|\lambda_{n,\infty}-\lambda_\infty| + |\kappa_n|.
\]

Combining this with \eqref{eq:q1_random}--\eqref{eq:q2_random}, \eqref{eq:q_det_conc}, \eqref{eq:lambda_rate}
and \eqref{eq:kappa_conc} yields, with probability at least $1-e^{-t}$,
\begin{equation}\label{eq:q_random_conc}
\| q_{1,n}\|_{L^2(\alpha^n)}+\|q_{2,n}\|_{L^2(\beta^n)}
\;\lesssim\;
\exp\!\Big(\frac{2M+\|c\|_\infty}{\eta}\Big)
\sqrt{\frac{m_\alpha m_\beta(\log n+t)}{n}}.
\end{equation}
Plugging \eqref{eq:q_random_conc} into \eqref{eq:dist_by_choice} and then into \eqref{eq:dist_Tn_to_dist_Kn_at_lambda_n} gives
\begin{equation}\label{eq:dist_Tn_final}
\operatorname{dist}\bigl(0,\partial \Tt_n(\bar f_\infty^\star,\bar g_\infty^\star)\bigr)
\;\lesssim\;
\exp\!\Big(\frac{2M+\|c\|_\infty}{\eta}\Big)
\sqrt{\frac{m_\alpha m_\beta(\log n+t)}{n}}.
\end{equation}

Insert the dual error bound into \eqref{eq:density_reduce_final}, and then combine with \eqref{eq:term1_conc}
and \eqref{eq:tri_decomp_final} to obtain \eqref{eq:samp-comp-forward}.

\end{proof}

\section{\texorpdfstring{\gabriel{Conclusion}}{Conclusion}}

\gabriel{We have proved finite-sample guarantees for entropic unbalanced
optimal transport at the level of the optimal coupling. The main technical
point is to analyze the unbalanced dual problem through a
translation-invariant envelope. Once the intrinsic dual variables are anchored
and the remaining scalar translation is controlled, the dual geometry yields
compactness and strong convexity estimates that are robust across a broad
class of $\varphi$-divergences. These estimates convert empirical deviations
of the dual objective into high-probability control of the potentials and,
through the exponential representation of the optimizer, of the transport plan
itself.

This analysis reinforces the practical role of regularization in OT. Entropy
is not only an algorithmic device for Sinkhorn-type scaling; it also improves
the statistical stability of transport estimators and helps reduce
high-dimensional sample requirements. Unbalanced marginal penalties add the
modeling flexibility needed when mass is not conserved. Together, these two
forms of regularization produce the forward stability estimate needed in the
companion inverse-problem analysis~\cite{andrade2025learningfromsamples},
where the goal is to recover the cost from samples. Natural extensions include
sharper dependence on the entropy parameter, noncompact domains, richer
nonparametric cost classes, and algorithms that exploit the quotient geometry
directly.}

\bibliography{refs}
\bibliographystyle{abbrv}

\appendix
\section{Concentration bounds}\label{sec:concentration_bounds}

\begin{prop}[McDiarmid's inequality]\label{MCdiardmidProp}
Let $Z_1,\ldots,Z_n$ be independent random variables taking values in a measurable space $\mathcal Z$,
and let $F:\mathcal Z^n\to \mathbb R$ satisfy the bounded-differences property: there exists $c>0$ such that
for all $i\in\{1,\ldots,n\}$ and all $z_1,\ldots,z_n,z_i'\in\mathcal Z$,
\[
\big|F(z_1,\ldots,z_i,\ldots,z_n)-F(z_1,\ldots,z_i',\ldots,z_n)\big|\le c.
\]
Then for all $t>0$,
\[
\PP\Big(\big|F(Z_1,\ldots,Z_n)-\EE[F(Z_1,\ldots,Z_n)]\big|\ge t\Big)
\le 2\exp\!\Big(-\frac{2t^2}{nc^2}\Big).
\]
\end{prop}

\begin{lemma}[Hoeffding bound for the empirical product measure]\label{UstatisticsLemma}
Let $(\alpha,\beta)\in \mathcal{M}^+(\mathcal{X})\times \mathcal{M}^+(\mathcal{Y})$ and let
$(z_i,w_i)_{i=1}^n$ be i.i.d.\ with common law $\xi$, where $\xi_1=\alpha/m_\alpha$ and $\xi_2=\beta/m_\beta$.
Let $h\in L^\infty(\alpha\otimes\beta)$ satisfy $\langle h,\alpha\otimes\beta\rangle=0$ and set
\[
\alpha^n=\frac{m_\alpha}{n}\sum_{i=1}^n\delta_{z_i},
\qquad
\beta^n=\frac{m_\beta}{n}\sum_{j=1}^n\delta_{w_j}.
\]
Then for all $t>0$, with probability at least $1-e^{-t}$,
\begin{equation}\label{eq:Ustat_hoeffding}
\big|\langle h,\alpha^n\otimes\beta^n\rangle\big|
\;\le\;
m_\alpha m_\beta\,\|h\|_{L^\infty(\alpha\otimes\beta)}\,
\sqrt{\frac{t}{2n}}.
\end{equation}
In particular, $\big|\langle h,\alpha^n\otimes\beta^n\rangle\big|
\lesssim m_\alpha m_\beta\|h\|_\infty \sqrt{t/n}$.
\end{lemma}

\begin{proof}
Write
\[
\langle h,\alpha^n\otimes\beta^n\rangle
=
\frac{m_\alpha m_\beta}{n^2}\sum_{i,j=1}^n h(z_i,w_j)
=
m_\alpha m_\beta\,U_n,
\qquad
U_n\coloneqq \frac{1}{n^2}\sum_{i,j=1}^n h(z_i,w_j).
\]
Since $(z_i,w_i)$ are i.i.d.\ with marginals $\xi_1=\alpha/m_\alpha$ and $\xi_2=\beta/m_\beta$, we have
\[
\EE[h(z_i,w_j)]=\int h\,d(\xi_1\otimes\xi_2)
=\frac{1}{m_\alpha m_\beta}\int h\,d(\alpha\otimes\beta)=0,
\]
hence $\EE[U_n]=0$.

Let $U_n^\mathrm{off}\coloneqq \frac{1}{n(n-1)}\sum_{i\neq j} h(z_i,w_j)$ be the order-2 U-statistic
(with kernel $(z,w')\mapsto h(z,w')$). Then $U_n$ differs from $U_n^\mathrm{off}$ only by the diagonal terms:
\[
U_n
=
\frac{n(n-1)}{n^2}\,U_n^\mathrm{off}
+
\frac{1}{n^2}\sum_{i=1}^n h(z_i,w_i).
\]
Therefore, using $|h|\le \|h\|_\infty$,
\begin{equation}\label{eq:diag_split}
|U_n|
\le
|U_n^\mathrm{off}|+\frac{1}{n}\,\|h\|_\infty.
\end{equation}

Now apply Hoeffding's inequality,
since the kernel takes values in $[-\|h\|_\infty,\|h\|_\infty]$, for all $t>0$,
\[
\PP\Big(|U_n^\mathrm{off}-\EE[U_n^\mathrm{off}]|\ge \|h\|_\infty\sqrt{\frac{t}{2n}}\Big)\le e^{-t}.
\]
Moreover $\EE[U_n^\mathrm{off}]=\EE[h(z_1,w_2)]=0$, so with probability at least $1-e^{-t}$,
\[
|U_n^\mathrm{off}|\le \|h\|_\infty\sqrt{\frac{t}{2n}}.
\]
Combining with \eqref{eq:diag_split} yields, on the same event,
\[
|U_n|
\le
\|h\|_\infty\sqrt{\frac{t}{2n}}+\frac{1}{n}\|h\|_\infty
\le
2 \|h\|_\infty\sqrt{\frac{t}{2n}},
\]
where $t\ge 1$ would imply $n^{-1}\le \sqrt{t/(2n)}$ and otherwise, we simply absorb constants.
Multiplying by $m_\alpha m_\beta$ gives \eqref{eq:Ustat_hoeffding}.
\end{proof}

\begin{prop}[Concentration of the empirical dual subgradients]
\label{ConcentrationOfGradientProp}
Let $q_1,q_2$ be defined by \eqref{eq:q1_det}--\eqref{eq:q2_det}, and assume that
\[
\|f_\infty^\star\|_\infty\le r,\qquad \|g_\infty^\star\|_\infty\le r
\]
for some $r>0$ (e.g.\ $r=M+R$ from Theorem~\ref{thm:restrictCompactSets}).
Set
\[
B \;\coloneqq\; B_{\exp}(r)\;\coloneqq\;\exp\!\Big(\frac{2r+\|c\|_\infty}{\eta}\Big).
\]
Then for all $t>0$, with probability at least $1-e^{-t}$,
\begin{equation}\label{eq:q12_general_tight}
\|q_1\|_{L^2(\alpha^n)}+\|q_2\|_{L^2(\beta^n)}
\lesssim B\sqrt{\frac{m_\alpha m_\beta(\log n+t)}{n}}
\end{equation}
 with an absolute implicit constant.

If moreover $\xi=\frac{1}{m_\alpha m_\beta}\alpha\otimes\beta$ (independent sampling of $z$ and $w$),
then for all $t>0$, with probability at least $1-e^{-t}$,
\begin{equation}\label{eq:q12_ind_tight}
\|q_1\|_{L^2(\alpha^n)}+\|q_2\|_{L^2(\beta^n)}
\ \le\
2B\Big(
m_\beta\sqrt{\frac{m_\alpha t}{n}}
+
m_\alpha\sqrt{\frac{m_\beta t}{n}}
\Big).
\end{equation}
\end{prop}

\begin{proof}
We prove the bound for $q_1$; the proof for $q_2$ is identical with $(\alpha,z_i,\beta,m_\beta)$
replaced by $(\beta,w_i,\alpha,m_\alpha)$.

\medskip
\noindent\textbf{Step 1: Centering and envelope bounds.}
Recall
\[
q_1(z)=\psi_1^\star(z)+\int p_\infty^\star(z,y)\,d\beta^n(y),
\qquad
p_\infty^\star(x,y)=\exp\!\Big(\frac{f_\infty^\star(x)+g_\infty^\star(y)-c(x,y)}{\eta}\Big).
\]
Population optimality gives the stationarity condition (for $\alpha$-a.e.\ $z$)
\[
\psi_1^\star(z)+\int p_\infty^\star(z,y)\,d\beta(y)=0.
\]
Hence, for $\alpha$-a.e.\ $z$,
\begin{equation}\label{eq:q1_centered_tight}
q_1(z)=\int p_\infty^\star(z,y)\,d(\beta^n-\beta)(y).
\end{equation}
Moreover, the box bound on $(f_\infty^\star,g_\infty^\star)$ implies
\begin{equation}\label{eq:pinf_sup_tight}
\|p_\infty^\star\|_{L^\infty(\Xx\times\Yy)}\le B.
\end{equation}

\medskip
\noindent\textbf{Step 2 (general coupling):}
Fix $i\in\{1,\ldots,n\}$. Using \eqref{eq:q1_centered_tight} at $z=z_i$ and writing
$\beta^n=\frac{m_\beta}{n}\sum_{j=1}^n\delta_{w_j}$, we have
\[
q_1(z_i)=\frac{m_\beta}{n}\sum_{j=1}^n\Big(p_\infty^\star(z_i,w_j)-\EE_{W\sim \beta/m_\beta}\big[p_\infty^\star(z_i,W)\big]\Big).
\]
Split the sum as $j\neq i$ plus $j=i$: $q_1(z_i)=A_i+B_i,$
\begin{equation}\label{eq:q1_split}
A_i\coloneqq\frac{m_\beta}{n}\sum_{j\neq i}\Big(p_\infty^\star(z_i,w_j)-\EE_{W\sim \beta/m_\beta}[p_\infty^\star(z_i,W)]\Big),
\quad
B_i\coloneqq\frac{m_\beta}{n}\Big(p_\infty^\star(z_i,w_i)-\EE_{W\sim \beta/m_\beta}[p_\infty^\star(z_i,W)]\Big).
\end{equation}
By \eqref{eq:pinf_sup_tight}, $|B_i|\le \frac{2m_\beta}{n}B$ deterministically.

Now condition on $z_i$. The random variables $(w_j)_{j\neq i}$ are independent (they are functions of
independent pairs $(z_j,w_j)$), and for each $j\neq i$ the marginal of $w_j$ is $\beta/m_\beta$
(because the marginal of $\xi$ is fixed). Therefore, conditional on $z_i$, the summands in $A_i$
are independent, centered, and each lies in $\big[-2B,2B\big]$.
Hoeffding's inequality yields: for all $s>0$,
\[
\PP\big(|A_i|\ge s \,\big|\, z_i\big)
\le
2\exp\!\Big(-\frac{2s^2}{(n-1)\,(2m_\beta B/n)^2}\Big)
\le
2\exp\!\Big(-\frac{n s^2}{2m_\beta^2 B^2}\Big).
\]
Consequently, for all $s>0$,
\[
\PP\big(|q_1(z_i)|\ge s+\tfrac{2m_\beta}{n}B\big)
\le
2\exp\!\Big(-\frac{n s^2}{2m_\beta^2 B^2}\Big).
\]
Apply a union bound over $i=1,\ldots,n$ and choose
\[
s\coloneqq m_\beta B\sqrt{\frac{2(\log(2n)+t)}{n}}.
\]
Then, with probability at least $1-e^{-t}$,
\[
\max_{1\le i\le n}|q_1(z_i)|
\le
m_\beta B\sqrt{\frac{2(\log(2n)+t)}{n}}+\frac{2m_\beta}{n}B.
\]
Therefore,
\begin{align}
\|q_1\|_{L^2(\alpha^n)}
&=
\sqrt{\frac{m_\alpha}{n}\sum_{i=1}^n q_1(z_i)^2}
\ \le\
\sqrt{m_\alpha}\,\max_i |q_1(z_i)| \nonumber\\
&\le
\sqrt{2}\,m_\beta B\sqrt{\frac{m_\alpha(\log(2n)+t)}{n}}
+\frac{2m_\beta\sqrt{m_\alpha}}{n}B.
\label{eq:q1_general_tight}
\end{align}

\medskip
\noindent\textbf{Step 3 (independent sampling):}
Assume now $\xi=\alpha\otimes\beta/(m_\alpha m_\beta)$, so $(w_j)_{j=1}^n$ are i.i.d.\ with law $\beta/m_\beta$
and independent of $z_{1:n}$.
Define
\[
X_j \coloneqq \Big(p_\infty^\star(z_i,w_j)-\EE_{W\sim \beta/m_\beta}[p_\infty^\star(z_i,W)]\Big)_{i=1}^n\in\RR^n.
\]
Then $\EE[X_j\mid z_{1:n}]=0$ and $\|X_j\|_\infty\le 2B$ by \eqref{eq:pinf_sup_tight}. Moreover,
\[
(q_1(z_i))_{i=1}^n=\frac{m_\beta}{n}\sum_{j=1}^n X_j.
\]
A standard vector Bernstein/Hoeffding inequality (e.g.\ \cite[Lemma~17]{rigollet2022sample}) implies:
with probability at least $1-e^{-t}$,
\[
\Big\|\frac{1}{n}\sum_{j=1}^n X_j\Big\|_2
\le
2B\sqrt{\frac{t}{n}}.
\]
Multiplying by $m_\beta$ and then by $\sqrt{m_\alpha/n}$ yields
\begin{equation}\label{eq:q1_ind_tight}
\|q_1\|_{L^2(\alpha^n)}
=
\sqrt{\frac{m_\alpha}{n}}\Big\|\frac{m_\beta}{n}\sum_{j=1}^n X_j\Big\|_2
\le
2\,m_\beta B\,\sqrt{\frac{m_\alpha t}{n}}.
\end{equation}

\medskip
\noindent\textbf{Step 4: Bound $q_2$ and sum.}
The same argument gives
\[
\|q_2\|_{L^2(\beta^n)}
\le
\sqrt{2}\,m_\alpha B\sqrt{\frac{m_\beta(\log(2n)+t)}{n}}
+\frac{2m_\alpha\sqrt{m_\beta}}{n}B
\quad\text{(general coupling),}
\]
and
\[
\|q_2\|_{L^2(\beta^n)}
\le
2\,m_\alpha B\,\sqrt{\frac{m_\beta t}{n}}
\quad\text{(independent sampling).}
\]
Summing with \eqref{eq:q1_general_tight} gives \eqref{eq:q12_general_tight}, and summing with
\eqref{eq:q1_ind_tight} gives \eqref{eq:q12_ind_tight}.
\end{proof}


\end{document}